\documentclass[11pt]{amsproc}
 \usepackage[margin=1in]{geometry}
\usepackage{setspace,fullpage}
\usepackage{graphicx}
\usepackage[nice]{nicefrac}
\usepackage{amssymb}
\usepackage{multirow}
\usepackage{array}
\usepackage{enumitem}

\usepackage{graphicx} % Required for \scalebox

\usepackage{amsmath,amsthm,amscd,amssymb, mathrsfs}

\usepackage{latexsym}

\numberwithin{equation}{section}

\theoremstyle{plain}
\newtheorem{theorem}{Theorem}[section]
\newtheorem{lemma}[theorem]{Lemma}

\newtheorem{proposition}[theorem]{Proposition}

\theoremstyle{definition}

\theoremstyle{remark}

\title{An improved non-linear Roth-type theorem in finite fields}
\author{Mark Lewko}

\address{Lebanon, New Hampshire USA}
\email{mlewko@gmail.com}
\begin{document}

\begin{abstract}
Let \(F\) be a finite field of odd characteristic. We prove that any set \(A\subset F\) with $|A|\geq C|F|^{5/6}$ contains a nontrivial quadratic progression $(x, x+y, x+y^2), y\neq 0.$
For prime fields, this improves the previous best-known exponent of \(7/8\), due to Kavrut and Wu. Unlike some of the previous papers, which rely on Katz's deep multivariate exponential-sum estimates, our argument uses only one-variable Weil-type estimates. We also construct, over certain non-prime finite fields, progression-free sets of size \(c|F|^{2/3}\). A key idea in the proof was suggested to the author by ChatGPT 5.5.
\end{abstract}

\maketitle

\section{Introduction}
We will use $F_q$ to denote a sufficiently large finite field of odd characteristic such that $|F_q|=q=p^s$ for some prime $p$. We will give a proof of the following theorem:

\begin{theorem}\label{thm:main}
Let $A \subset F_q$. There exists a constant $C$ such that for all $|A| \geq C q^{\frac{5}{6}}$, $A$ contains a quadratic progression of the form $(x,x+y,x+y^2)$ with $y\neq 0$.
\end{theorem}

In fields of prime order, a quantitative result of this form was first proven for sets of size $|A| \geq c q^{\frac{14}{15}}$ by Bourgain and Chang \cite{BC} in 2016. This was improved to $|A| \geq c q^{\frac{11}{12}}$ by Dong, Li and Sawin \cite{DLS} and more recently to $|A| \geq c q^{\frac{7}{8}}$ by Kavrut and Wu \cite{KW}. The latter two papers rely on deep multivariate exponential-sum estimates of Katz, ultimately derived from Deligne's proof of the Riemann Hypothesis for varieties. Our argument avoids these higher-dimensional estimates and uses only a one-dimensional Weil-type estimate. We also note that Peluse obtained polynomial savings for longer polynomial configurations in a remarkable paper \cite{S}.

As with all prior work on this problem, we proceed by proving an estimate on a related bilinear averaging operator. In order to state this estimate we must introduce some notation. Let $e(\cdot): F_{q} \rightarrow \mathbb{C}$ denote a non-trivial additive character, and let $f: F_q \rightarrow \mathbb{C}$ denote a complex-valued function on $F_q$. We define
$$
\mathbb{E}[f] = \frac{1}{q} \sum_{x \in F_q} f(x)
$$
and the corresponding norms
$$
\|f\|_{r} = \left( \frac{1}{q} \sum_{x \in F_q}|f(x)|^r \right)^{1/r}.
$$
We will also use
$$
\|g\|_{\ell^2}=\left(\sum_{\xi\in F_q}|g(\xi)|^2\right)^{1/2}.
$$
We define the Fourier transform of $f$ by
$$
\widehat f(\xi):=\mathbb E_x f(x)e(-x\xi)
=\frac1q\sum_{x\in F_q}f(x)e(-x\xi).
$$
The Fourier inversion formula then takes the form
$$
f(x)=\sum_{\xi\in F_q}\widehat f(\xi)e(x\xi).
$$
With this convention, Parseval's identity states
$$
\|f\|_2=\|\widehat f\|_{\ell^2}.
$$
We also use the usual asymptotic notation $X\lesssim Y$ to indicate that there exists a universal constant $C$ such that $X \leq C Y$.

For $f_1,f_2 : F_{q} \rightarrow \mathbb{C} $ we consider the following operator
$$
\mathcal{A}(f_1,f_2)(x) := \frac{1}{q} \sum_{y \in F_q} f_1(x+y)f_2(x+y^2).
$$
We will prove:

\begin{theorem}\label{thm:bilinear}
Let $f_1,f_2 : F_{q} \rightarrow \mathbb{C}$. Then
$$
\|\mathcal{A}(f_1,f_2) - \mathbb{E}[f_1]\mathbb{E}[f_2]\|_{2} \lesssim q^{-\delta} \|f_1\|_{2} \|f_2\|_{2}
$$
holds for $\delta = \frac{1}{4}$.
\end{theorem}

The proof of the main result follows from this using the following:

\begin{lemma}\label{lem:reduce}
Assume that Theorem \ref{thm:bilinear} holds with a factor of $q^{-\delta}$, where $0<\delta<\frac34$. Then Theorem \ref{thm:main} holds for sets of size $|A| \gtrsim q^{1-\frac{2}{3}\delta}$. In particular, Theorem \ref{thm:bilinear} implies Theorem \ref{thm:main}.
\end{lemma}

\begin{proof}
Let $|A|=\alpha q$ and let $f=1_A$. We then have
$$
\mathbb{E}_x\mathbb{E}_y[f(x)f(x+y)f(x+y^2)]
= \mathbb{E}_x [f(x) \mathcal{A}(f,f)(x)]
$$
$$
= \mathbb{E}_x [f(x) \mathbb{E}[f]^2] 
+  \mathbb{E}_x [f(x) \left( \mathcal{A}(f,f)(x) -\mathbb{E}[f]^2 \right) ] 
$$
$$
\geq \mathbb{E}[f]^3 
-\|f\|_{2} \cdot \|\mathcal{A}(f,f) - \mathbb{E}[f]^2 \|_{2}
$$
$$
\geq \alpha^{3} -  C q^{-\delta} \alpha^{\frac{3}{2}}.
$$
The contribution to this average from the value $y=0$ is $\alpha/q$, and therefore the average over $y\neq 0$ is positive provided $\alpha^3-Cq^{-\delta}\alpha^{3/2}-\alpha/q>0$. This holds, after increasing the implicit constant if necessary, whenever $\alpha \gtrsim q^{-\frac{2}{3}\delta}$, since $\delta<\frac34$. Therefore there exist $x\in F_q$ and $y\neq 0$ such that $x,x+y,x+y^2\in A$. This proves the claim.
\end{proof}

\section{Acknowledgments and Use of Artificial Intelligence}
I would like to thank Shukun Wu for pointing out a serious error in a much earlier manuscript, and Emmanuel Kowalski and Will Sawin for answering various questions about multivariate exponential sums related to this and other projects.

Several years ago, I found a cruder form of this argument in which the operators $T_h$ were estimated with successive applications of the $TT^{*}$ method. This was enough to recover the Kavrut--Wu exponent, while avoiding multivariate exponential sums and leading to a considerably simpler proof. I also noticed that the argument could be iterated arbitrarily many times, provided one could obtain square-root cancellation in a certain family of exponential sums in increasingly many variables. These sums were singular in the sense of Katz, and the needed estimates do not appear to follow directly from the existing literature. If carried through, this method had the potential to establish the main result for sets of size $c_\epsilon q^{5/6 +\epsilon}$, or slightly weaker than what we have obtained here. During a discussion with ChatGPT 5.5 about how one might control these sums, it pointed out that the kernel could be represented as a multiplicative convolution, which is a key insight used in the proof presented here.

\section{Reduction to a linear estimate}

We now turn to the proof of Theorem \ref{thm:bilinear}. The initial steps of the argument will follow \cite{DLS}. By Fourier inversion (see (2.1) in \cite{BC}), we have
$$
\mathcal{A}(f_1,f_2)(x) = \mathbb{E}_y [f_1(x+y)f_2(x+y^2)]
$$
\begin{equation}\label{eq:Afour}
=
\sum_{n_1,n_2 \in F_q}
\widehat{f}_1(n_1) \widehat{f}_2(n_2)
e((n_1+n_2)x) K(n_1,n_2),
\end{equation}
where
\begin{equation}\label{eq:Kdef}
K(a,b):= \mathbb{E}_y[e(ay+by^2)]
=
\left\{
\begin{array}{ll}
\sigma_F q^{-1/2}\chi(b)e\left(-\frac{a^2}{4b}\right)
\quad&\mbox{for}~~ b\neq 0, \\[0.5em]
1 \quad&\mbox{for}~~ a = 0 \text{ and } b = 0, \\[0.5em]
0 \quad&\mbox{for}~~ a\neq 0 \text{ and } b= 0.
\end{array}
\right.
\end{equation}
Here $\chi(\cdot)$ denotes the quadratic multiplicative character and $\sigma_F$ is a complex number such that $|\sigma_F|=1$. See Theorem 5.22 in \cite{LN}.

From \eqref{eq:Afour} we have
$$
\mathcal{A}(f_1,f_2)(x)
=
\mathbb{E}[f_1]\mathbb{E}[f_2]
+
\sum_{\substack{n_1,n_2 \in F_q \\ (n_1,n_2) \neq (0,0)}}
\widehat{f}_1(n_1) \widehat{f}_2(n_2)
e((n_1+n_2)x) K(n_1,n_2).
$$
Furthermore, if $n_2=0$ and $n_1\neq 0$, then $K(n_1,0)=0$. Hence
$$
\mathcal{A}(f_1,f_2)(x)
=
\mathbb{E}[f_1]\mathbb{E}[f_2]
+
\sum_{\substack{n_1,n_2 \in F_q \\ n_2 \neq 0}}
\widehat{f}_1(n_1) \widehat{f}_2(n_2)
e((n_1+n_2)x) K(n_1,n_2).
$$
Performing the change of variables $m=n_1+n_2$ and $n=n_2$, we obtain
$$
\mathcal{A}(f_1,f_2)(x)
=
\mathbb{E}[f_1]\mathbb{E}[f_2]
+
\sum_{m\in F_q}
\left(
\sum_{n\neq 0}
\widehat{f}_1(m-n)\widehat{f}_2(n)K(m-n,n)
\right)e(mx).
$$
From Parseval's identity, we have
$$
\|\mathcal{A}(f_1,f_2)-\mathbb{E}[f_1]\mathbb{E}[f_2]\|_2
=
\left(
\sum_{m\in F_q}
\left|
\sum_{n\neq 0}
\widehat{f}_1(m-n)\widehat{f}_2(n)K(m-n,n)
\right|^2
\right)^{1/2}.
$$
It thus suffices to show that
\begin{equation}\label{eq:reduction}
\sum_{m\in F_q}
\left|
\sum_{n\neq 0}
\widehat{f}_1(m-n)\widehat{f}_2(n)K(m-n,n)
\right|^2
\lesssim
q^{-\frac{1}{2}}
\|\widehat{f}_1\|_{\ell^2}^2
\|\widehat{f}_2\|_{\ell^2}^2.
\end{equation}

Expanding the square on the left side, we have
$$
\sum_{m\in F_q}
\sum_{\substack{n_1,n_2 \in F_q \\ n_1,n_2 \neq 0}}
\widehat{f}_1(m-n_1)
\overline{\widehat{f}_1(m-n_2)}
\widehat{f}_2(n_1)
\overline{\widehat{f}_2(n_2)}
K(m-n_1,n_1)
\overline{K(m-n_2,n_2)}.
$$
Using the change of variables $n_1=v$, $n_2=v+h$, and $m=u+v$, this quantity can be rewritten as
$$
\sum_{h\in F_q}
\sum_{\substack{u,v\in F_q \\ v\neq 0 \\ v+h\neq 0}}
\widehat{f}_1(u)
\overline{\widehat{f}_1(u-h)}
\widehat{f}_2(v)
\overline{\widehat{f}_2(v+h)}
K(u,v)
\overline{K(u-h,v+h)}.
$$
Defining
$$
F_h(u):=\widehat{f}_1(u)\overline{\widehat{f}_1(u-h)}
$$
and
$$
G_h(v):=\widehat{f}_2(v)\overline{\widehat{f}_2(v+h)},
$$
we can rewrite the above as
\begin{equation}\label{eq:FG}
\sum_{h\in F_q}
\sum_{\substack{u,v\in F_q \\ v\neq 0 \\ v+h\neq 0}}
F_h(u)G_h(v)K(u,v)\overline{K(u-h,v+h)}.
\end{equation}

We next consider the inner sum in the case when $h=0$. In this case we have
$$
K(u,v)\overline{K(u-h,v+h)}=|K(u,v)|^2=q^{-1}
$$
for $v\neq 0$. Thus the $h=0$ contribution is bounded by
$$
q^{-1}\|F_0\|_{\ell^1}\|G_0\|_{\ell^1}
\leq
q^{-1}
\|\widehat{f}_1\|_{\ell^2}^2
\|\widehat{f}_2\|_{\ell^2}^2.
$$

We now claim the following proposition, which we will prove later.

\begin{proposition}\label{prop:slice}
Let $G: F_q \rightarrow \mathbb{C}$ and let $h\neq0$. Define
$$
T_h(G)(u)
:=
\sum_{\substack{v \in F_q \\ v \neq 0 \\ v+h \neq 0}}
G(v)K(u,v)\overline{K(u-h,v+h)}.
$$
Then
$$
\|T_h(G)\|_{\ell^2}
\lesssim
q^{-\frac{1}{2}}\|G\|_{\ell^2}.
$$
\end{proposition}

Returning to \eqref{eq:FG}, applying Cauchy--Schwarz and then Proposition \ref{prop:slice}, we may bound the contribution from $h\neq 0$ by
$$
\sum_{\substack{h\in F_q \\ h\neq 0}}
\|F_h\|_{\ell^2}\|T_h(G_h)\|_{\ell^2}
\lesssim
q^{-\frac{1}{2}}
\sum_{h\in F_q}
\|F_h\|_{\ell^2}\|G_h\|_{\ell^2}.
$$
By Cauchy--Schwarz in $h$,
$$
\sum_{h\in F_q}
\|F_h\|_{\ell^2}\|G_h\|_{\ell^2}
\leq
\left(\sum_h\|F_h\|_{\ell^2}^2\right)^{1/2}
\left(\sum_h\|G_h\|_{\ell^2}^2\right)^{1/2}.
$$
Since
$$
\sum_h\|F_h\|_{\ell^2}^2
=
\sum_{h,u}
|\widehat f_1(u)|^2|\widehat f_1(u-h)|^2
=
\|\widehat f_1\|_{\ell^2}^4,
$$
and similarly
$$
\sum_h\|G_h\|_{\ell^2}^2
=
\|\widehat f_2\|_{\ell^2}^4,
$$
we conclude that the quantity in \eqref{eq:FG} is bounded by
$$
q^{-\frac{1}{2}}
\|\widehat f_1\|_{\ell^2}^2
\|\widehat f_2\|_{\ell^2}^2.
$$
This proves \eqref{eq:reduction}, and hence Theorem \ref{thm:bilinear}, subject to Proposition \ref{prop:slice}.

\section{Proof of Proposition \ref{prop:slice}}

\begin{proof}
We now prove Proposition \ref{prop:slice}. Recall that
$$
T_h(G)(u)
=
\sum_{\substack{v \in F_q \\ v\neq 0 \\ v+h\neq 0}}
G(v)K(u,v)\overline{K(u-h,v+h)}.
$$
From \eqref{eq:Kdef}, for $v\neq 0,-h$, we have
$$
K(u,v)\overline{K(u-h,v+h)}
=
q^{-1}\chi(v)\chi(v+h)
e\left(
-\frac{u^2}{4v}
+
\frac{(u-h)^2}{4(v+h)}
\right).
$$
The factor $\chi(v)\chi(v+h)$ has modulus one and is independent of $u$, so it may be absorbed into $G(v)$. Therefore it is enough to prove
$$
\|\widetilde T(g)\|_{\ell^2}
\lesssim
q^{1/2}\|g\|_{\ell^2},
$$
where
$$
\widetilde T(g)(x)
=
\sum_{\substack{y\in F_q\\ y\neq0,\ y\neq-h}}
g(y)
e\left(
-\frac{x^2}{4y}
+
\frac{(x-h)^2}{4(y+h)}
\right).
$$
Since $4$ is invertible in $F_q$, we may replace the additive character by $e(\cdot/4)$ and suppress the factor of $4$. Thus it is enough to prove the same estimate for
$$
\widetilde T(g)(x)
=
\sum_{\substack{y\in F_q\\ y\neq0,\ y\neq-h}}
g(y)
e\left(
-\frac{x^2}{y}
+
\frac{(x-h)^2}{y+h}
\right).
$$

Squaring and summing in $x$, we obtain
$$
\|\widetilde T(g)\|_{\ell^2}^2
=
\sum_{\substack{y,z\in F_q\\ y,z\neq0,-h}}
g(y)\overline{g(z)}B_h(y,z),
$$
where
$$
B_h(y,z)
=
\sum_{x\in F_q}
e\left(
-\frac{x^2}{y}
+
\frac{(x-h)^2}{y+h}
+
\frac{x^2}{z}
-
\frac{(x-h)^2}{z+h}
\right).
$$
It is therefore enough to prove
$$
\left|
\sum_{\substack{y,z\in F_q\\ y,z\neq0,-h}}
g(y)\overline{g(z)}B_h(y,z)
\right|
\lesssim
q\|g\|_{\ell^2}^2.
$$

We first evaluate $B_h(y,z)$. The phase is a quadratic polynomial in $x$. The coefficient of $x^2$ is
$$
A=
\frac{h(y-z)(h+y+z)}{(h+y)(h+z)yz},
$$
the coefficient of $x$ is
$$
B=
\frac{2h(y-z)}{(h+y)(h+z)},
$$
and the constant term is
$$
C=
\frac{h^2(z-y)}{(h+y)(h+z)}.
$$
Thus, if $y\neq z$ and $h+y+z\neq0$, using the evaluation of a Gauss sum gives
$$
B_h(y,z)
=
\sigma_F q^{1/2}
\chi\left(
\frac{h(h+y+z)(y-z)}
{(h+y)(h+z)yz}
\right)
e\left(
\frac{h(z-y)}{h+y+z}
\right),
$$
where $|\sigma_F|=1$. If $y=z$, then $B_h(y,y)=q$. If $h+y+z=0$ and $y\neq z$, then the quadratic coefficient vanishes but the linear coefficient is nonzero, and hence $B_h(y,z)=0$.

Now put
$$
c=\frac h2,\qquad Y=y+c,\qquad Z=z+c.
$$
Then
$$
h+y+z=Y+Z,\qquad y-z=Y-Z,
$$
and
$$
(h+y)y=(Y+c)(Y-c)=Y^2-c^2.
$$
The restrictions $y\neq0,-h$ become $Y\neq c,-c$. In these variables, away from the two lines $Y=Z$ and $Y=-Z$, the formula above becomes
$$
B_h(Y-c,Z-c)
=
\sigma_F q^{1/2}
\chi\left(
\frac{h(Y^2-Z^2)}
{(Y^2-c^2)(Z^2-c^2)}
\right)
e\left(
\frac{h(Z-Y)}{Y+Z}
\right).
$$

We now remove the denominator inside the quadratic character. Define
$$
D(Y)=\chi(Y^2-c^2).
$$
For $Y\neq\pm c$, we have $D(Y)=\pm1$. Thus, replacing $f(Y)$ by $f(Y)D(Y)$ does not change $|f(Y)|$ or $\|f\|_{\ell^2}$. Therefore it is enough to bound the same quadratic form after multiplying the kernel by $D(Y)D(Z)$. Define
$$
B_{h,0}(Y,Z)
=
D(Y)B_h(Y-c,Z-c)D(Z).
$$
For $Y,Z\neq\pm c$, away from the two lines $Y=Z$ and $Y=-Z$, this is
$$
B_{h,0}(Y,Z)
=
\omega_h q^{1/2}
\chi(Y^2-Z^2)
e\left(
\frac{h(Z-Y)}{Y+Z}
\right),
$$
where $\omega_h=\sigma_F\chi(h)$ and $|\omega_h|=1$. Also
$$
B_{h,0}(Y,Y)=q,
$$
and, for $Y\neq0$,
$$
B_{h,0}(Y,-Y)=0.
$$

We next treat the terms with $Y,Z\neq0$. We exclude the values $Y=\pm c$ and extend the relevant functions by zero at these two points. Put
$$
r=\frac ZY.
$$
Then
$$
Y^2-Z^2=Y^2(1-r^2),
$$
so
$$
\chi(Y^2-Z^2)=\chi(1-r^2).
$$
Also
$$
\frac{Z-Y}{Y+Z}
=
\frac{r-1}{r+1}.
$$
Define
$$
L_h(r)
:=
\omega_h\chi(1-r^2)
e\left(
h\frac{r-1}{r+1}
\right)
$$
for $r\neq\pm1$, and set $L_h(1)=L_h(-1)=0$. Then, for $Y,Z\neq0$, we have
$$
B_{h,0}(Y,Z) = q\,1_{Y=Z} + q^{1/2}L_h(Z/Y).
$$

We will need a variant of Weil's mixed exponential sum estimate \cite{Weil}. We use the following formulation stated by Cochrane and Pinner in \cite{CP}.  In what follows it will be important that only the number of zeros and poles of the rational function in the multiplicative character appear in the bound and not, say, the degree.

\begin{theorem}\label{thm:mixed}
Let $\psi$ be a non-trivial additive character of $F_q$, let $\rho$ be a multiplicative character of $F_q^\times$, and let $P\in F_q[X]$ be a nonconstant polynomial of degree $d$ with $p\nmid d$. Let $G\in F_q(X)$, and let $S$ be the set of finite points at which $G$ has a zero or a pole. Put $\ell=|S|$. Then
$$
\left|
\sum_{x\in F_q\setminus S}
\rho(G(x))\psi(P(x))
\right|
\lesssim
(d+\ell)q^{1/2}.
$$
Here the implicit constant is absolute. In particular, it is independent of the orders of the zeros and poles of $G$.
\end{theorem}

We will use the following consequence.

\begin{lemma}\label{lem:specific-mixed}
Let $\eta$ be a multiplicative character of $F_q^\times$, let $e(\cdot)$ be a nontrivial additive character, and let $\lambda\neq 0$. Then
$$
\left|
\sum_{r\neq 0,\pm1}
\eta(r)\chi(1-r^2)
e\!\left(\lambda\frac{r-1}{r+1}\right)
\right|
\lesssim q^{1/2}.
$$
Consequently,
$$
\left|
\sum_{r\in F_q^\times}L_h(r)\eta(r)
\right|
\lesssim q^{1/2}.
$$
\end{lemma}

\begin{proof}
We first reduce the two multiplicative characters to one character. The character group of $F_q^\times$ is cyclic, so we may choose a character $\rho$ which generates the subgroup generated by $\eta$ and $\chi$. Thus, for some integers $a,b$,
$$
\eta=\rho^a,
\qquad
\chi=\rho^b.
$$
This includes the case where both characters are trivial.

Now make the fractional linear change of variables
$$
s=\frac{r-1}{r+1},
\qquad
r=\frac{1+s}{1-s}.
$$
The excluded values $r=0,\pm1$ correspond to deleting only the finite values $s=-1,0,1$. Moreover
$$
1-r^2
=
1-\left(\frac{1+s}{1-s}\right)^2
=
\frac{-4s}{(1-s)^2}.
$$
Therefore, for the remaining values of $s$,
$$
\eta(r)\chi(1-r^2)
=
\rho\!\left(
C s^b(s+1)^a(s-1)^{-a-2b}
\right)
$$
for some constant $C\in F_q^\times$. Thus the sum is of the form
$$
\sum_{s\neq -1,0,1}
\rho(G(s))e(\lambda s),
$$
where
$$
G(s)=C s^b(s+1)^a(s-1)^{-a-2b}.
$$

The exponents $a,b$ may depend on $\eta$ and $\chi$, but this does not affect the constant. The only possible finite zeros and poles of $G$ are $-1, 0,1$.

The powers with which these factors occur may change, but the set of possible factors does not.

The additive phase is the linear polynomial $P(s)=\lambda s$. Since $\lambda\neq0$, it has degree $1$, and $p\nmid 1$. Hence Theorem \ref{thm:mixed} applies with $d=1$ and $\ell\leq3$. If one of the three points $-1,0,1$ is not actually a zero or pole of $G$, deleting it changes the sum by only $O(1)$. We conclude that
$$
\left|
\sum_{r\neq 0,\pm1}
\eta(r)\chi(1-r^2)
e\!\left(\lambda\frac{r-1}{r+1}\right)
\right|
\lesssim q^{1/2}.
$$
Taking $\lambda=h$ and absorbing the unimodular factor $\omega_h$ gives the stated bound for $L_h$.
\end{proof}

Recalling that $B_{h,0}(Y,Z) = q\,1_{Y=Z} + q^{1/2}L_h(Z/Y)$ our main task will be, for $f:F_q^\times\rightarrow \mathbb{C}$, to estimate
$$
\sum_{Y,Z\in F_q^\times}
f(Y)\overline{f(Z)}L_h(Z/Y).
$$
The key insight is that this is a multiplicative convolution. For a multiplicative character $\eta$ of $F_q^\times$, we define the multiplicative Fourier coefficients  
$$
M_f(\eta)=\sum_{Y\in F_q^\times}f(Y)\overline{\eta(Y)}.
$$
This allows us to represent
$$
f(Y)=\frac1{q-1}\sum_{\eta}M_f(\eta)\eta(Y).
$$
From orthogonality we have the following Parseval-type identity 
\begin{equation}\label{eq:mulP}
\frac1{q-1}\sum_{\eta}|M_f(\eta)|^2
=
\sum_{Y\in F_q^\times}|f(Y)|^2.
\end{equation}
Writing $Z=rY$, we have
$$
\sum_{Y,Z\in F_q^\times} f(Y)\overline{f(Z)}L_h(Z/Y)
= \sum_{r\in F_q^\times}L_h(r) \sum_{Y\in F_q^\times}f(Y)\overline{f(rY)}.
$$

Using the character expansion,
$$
\sum_{Y\in F_q^\times}f(Y)\overline{f(rY)}
=
\frac1{q-1}
\sum_{\eta}|M_f(\eta)|^2\overline{\eta(r)}.
$$
Therefore
$$
\sum_{Y,Z\in F_q^\times}
f(Y)\overline{f(Z)}L_h(Z/Y)
=
\frac1{q-1}
\sum_{\eta}
|M_f(\eta)|^2
\sum_{r\in F_q^\times}L_h(r)\overline{\eta(r)}.
$$
By Lemma \ref{lem:specific-mixed} and \eqref{eq:mulP} this is bounded in absolute value by
$$
q^{1/2}
\frac1{q-1}
\sum_{\eta}|M_f(\eta)|^2
=
q^{1/2}
\sum_{Y\in F_q^\times}|f(Y)|^2.
$$
Consequently,
$$
\left|
\sum_{Y,Z\in F_q^\times}
f(Y)\overline{f(Z)}B_{h,0}(Y,Z)
\right|
\lesssim
q\sum_{Y\in F_q^\times}|f(Y)|^2.
$$

It remains to include the terms with $Y=0$ or $Z=0$, corresponding to $y=-h/2$ or $z=-h/2$. The term $Y=Z=0$ is bounded by $q|f(0)|^2$. If exactly one of $Y,Z$ is zero, then the Gauss sum evaluation above gives
$$
|B_{h,0}(Y,Z)|\lesssim q^{1/2}.
$$
Thus these terms are bounded by
$$
q|f(0)|^2
+
q^{1/2}|f(0)|\sum_{Y\neq0}|f(Y)|.
$$
By Cauchy--Schwarz,
$$
q^{1/2}|f(0)|\sum_{Y\neq0}|f(Y)|
\leq
q|f(0)|\left(\sum_{Y\neq0}|f(Y)|^2\right)^{1/2}
\leq
q\sum_Y |f(Y)|^2.
$$
Therefore the terms with $Y=0$ or $Z=0$ contribute
$$
\lesssim q\sum_Y |f(Y)|^2.
$$
Combining this with the bound for $Y,Z\neq0$, we obtain
$$
\left|
\sum_{\substack{Y,Z\in F_q\\ Y,Z\neq\pm c}}
f(Y)\overline{f(Z)}B_{h,0}(Y,Z)
\right|
\lesssim
q\|f\|_{\ell^2}^2.
$$
Undoing the multiplication by $D(Y)$, and then undoing the change of variables $Y=y+h/2$, gives
$$
\left|
\sum_{\substack{y,z\in F_q\\ y,z\neq0,-h}}
g(y)\overline{g(z)}B_h(y,z)
\right|
\lesssim
q\|g\|_{\ell^2}^2.
$$
This implies $\|\widetilde T(g)\|_{\ell^2}
\lesssim
q^{1/2}\|g\|_{\ell^2}.
$
Returning to $T_h$, which includes a factor of $q^{-1}$, we obtain
$$
\|T_h(G)\|_{\ell^2}
\lesssim
q^{-1/2}\|G\|_{\ell^2}.
$$
This proves Proposition \ref{prop:slice}.
\end{proof}

\section{Lower bounds and examples}
We close by recording some lower bound constructions. These examples suggest
that the exponent $\frac{2}{3}$ is a natural barrier for this problem.

In their paper, Bourgain and Chang observed a lower bound for the bilinear operator
estimate appearing above. In the notation of this paper, their example shows that one
cannot obtain an estimate of the form
$$
\|\mathcal{A}(f_1,f_2)-\mathbb{E}[f_1]\mathbb{E}[f_2]\|_2
    \lesssim q^{-\delta}\|f_1\|_2\|f_2\|_2
$$
with $\delta=\frac12$ or larger. By the reduction in Lemma
\ref{lem:reduce}, the exponent $\delta=\frac12$ would correspond to
the threshold
$$
        |A| \gtrsim q^{1-\frac23\cdot\frac12}
        =
        q^{2/3}.
$$
However, the Bourgain--Chang example is function-theoretic. It gives an
obstruction to the operator estimate, but it does not appear to give a
combinatorial construction of a set $A\subset F_q$ of size about $q^{2/3}$
with no configurations of the form $(x,x+y,x+y^2)$. We give such a construction for certain finite fields below.

First, let us note that a simple greedy argument gives a lower bound of
$\gtrsim q^{1/2}$ in every finite field. Suppose that $A$ contains no
nontrivial configurations of the form $(x,x+y,x+y^2)$. We ask how many new
elements are forbidden from being added to $A$.

If adding a new element creates a configuration, then that new element must
occupy one of the three positions $(x,x+y,x+y^2)$ while the other entries are already in $A$. Once two entries are fixed, there are only $O(1)$ possible values for the remaining entry. Indeed, the parameter $y$ is then determined by either a linear equation or a quadratic equation.
Thus each pair of elements of $A$ forbids at most $O(1)$ new elements. Thus one can proceed with a greedy construction until the set has size
$$
        |A| \gtrsim q^{1/2}.
$$

For quadratic extensions $K=\mathbb F_{q^2}$, an explicit construction can be given. Let
$K/F$ be a quadratic extension, with $|F|=q$ and $q$ odd. Choose
$\omega\in K$ with $\omega^2\in F$ and $\omega\notin F$, and consider the
$F$-line
$$
        L=\{\omega a:a\in F\}\subset K.
$$
Then $|L|=q=|K|^{1/2}$. If $0\neq y=\omega a\in L$, then
$$
        y^2=\omega^2 a^2\in F^\times.
$$
But $L\cap F=\{0\}$, so $y^2\notin L$. Hence if $\{x, x+y, x+y^2\}$ all lie in $L$, then subtracting $x$ gives $y,y^2\in L$ which forces $y=0$. Thus $L$ contains no nontrivial configurations.

The same idea can be improved in cubic extensions. Instead of taking a line in
a quadratic extension, we look for a plane in a cubic extension with the
property that no nonzero element of the plane has its square again in the
plane.

\begin{theorem}\label{thm:cubiclb}
Let $F=\mathbb F_q$ be a finite field of odd order, and let $K/F$ be a cubic
extension. Then there exists a two-dimensional $F$-linear subspace
$V\subset K$ such that if $y\in V,\quad y^2\in V$ then $y=0$. Consequently, $V$ has size $|V|=q^2=|K|^{2/3}$, and $V$ contains no nontrivial configurations of the form $(x, x+y, x+y^2)$.
\end{theorem}

\begin{proof}
We proceed by a counting argument.  Throughout the proof, a plane means a two-dimensional
$F$-linear subspace of the three-dimensional $F$-space $K$.

First, the total number of planes in $K$ is
$$
        \frac{(q^3-1)(q^3-q)}{(q^2-1)(q^2-q)}
        =
        q^2+q+1.
$$
Indeed, the numerator counts ordered pairs of linearly independent vectors in
$K$, and the denominator counts the number of ordered bases given rise to a fixed plane.

Now count the planes containing $1$. Such a plane is spanned by $1$ and some
element of $K\setminus F$. There are $q^3-q$ choices for this second element.
For a fixed plane containing $1$, exactly $q^2-q$ of its elements lie outside
$F$. Hence the number of planes containing $1$ is
$$ \frac{q^3-q}{q^2-q}   = q+1.$$
It follows that the number of planes $V\subset K$ with $ 1\notin V$ is $(q^2+q+1)-(q+1)=q^2$.

Call a plane $V$ bad if $1\notin V$ and there is some nonzero $y\in V$ such that $y^2\in V$. We will show that not all of the planes can be bad.

Next notice that if $1 \notin V$ then since $V$ is $F$-linear we cannot have $y \in V$ for $y \in F^\times$ since that would imply $1=y^{-1}y\in V,$ contradicting that $1\notin V$. Thus any candidate for $y$ and $y^2$ being in $V$ must satisfy $y \in K\setminus F$. Fix $y\in K\setminus F$. Since $K/F$ has degree $3$, the elements $\{1,y,y^2\}$ are linearly independent over $F$. Therefore $\{y,y^2\}$ must span $V$ and every bad plane arises in this way. Indeed, if $V$ is bad, then for some nonzero $y\in V$ we have $y^2\in V$. We now show that each bad plane is produced by many such $y$. Indeed if $\{y,y^2\}$ span $V$ then $\{ay,a^2y^2\}$ also span $V$ for all $a \in F^\times$, or for $q-1$ scalar multiples.

Next we notice that since $K$ is a cubic extension, we may write $y^3 =Ay^2+By+C$ for $A,B,C \in F$ and $C\neq 0$. If $y$ and $y^2$ are in $V$ then $z:= y^2 - \frac A2 y \in V$. Now we use that the characteristic is odd here. Further notice that $z$ is not an $F$-multiple of $y$ since that would imply that $y^2$ is an $F$-multiple of $y$ forcing $y\in F$. However notice that $z^2 = y^4 -Ay^3 +A^2/4 y^2.$ Using $y^3=Ay^2+By+C$ we have that $y^4 = y(Ay^2+By+C) = Ay^3+By^2+Cy$ which implies $z^2  =    \left(B+\frac{A^2}{4}\right)y^2+Cy$. So $z^2$ is in the span of $\{y, y^2\}$ and thus also in $V$. Thus if we consider the association of $y \in K \setminus F$ to bad planes, there are at least $2(q-1)$ choices that produce the same plane, while there are $q^3 -q$ choices of $y \in K \setminus F$. So the number of bad planes is at most $\frac{q^3-q}{2(q-1)}  =  \frac{q(q+1)}2$ while we have previously established the number of $F$-linear planes not containing $1$ is $q^2$. Therefore a plane with the desired properties must exist. The fact that this set does not contain a nontrivial configuration of the form $(x,x+y,x+y^2)$ follows as before.
\end{proof}

\bibliographystyle{amsplain}

\end{document}